\documentclass[12pt,a4paper]{article}
\usepackage[utf8]{inputenc}
\usepackage[english]{babel}
\usepackage[T1]{fontenc}
\usepackage{amssymb}
\usepackage{amsmath}
\usepackage{amsthm}
\usepackage{bm}
\usepackage{dsfont}

\usepackage{tikz}
\usepackage{stmaryrd}

\newtheorem{theorem}{Theorem}[section]

\newtheorem{proposition}[theorem]{Proposition}
\newtheorem{corollary}[theorem]{Corollary}

\theoremstyle{definition}
\newtheorem{assumption}[theorem]{Assumption}
\newtheorem{definition}[theorem]{Definition}

\theoremstyle{remark}
\newtheorem{remark}[theorem]{Remark}

\newtheoremstyle{problem}
     {\topsep}
     {\topsep}
     {\slshape}
     {}
     {\bfseries}
     {.}
     {.5em}
     {}

\theoremstyle{problem}
\newtheorem{problem}[theorem]{Problem}

\numberwithin{equation}{section}

\newcommand{\cart}{\times}
\newcommand{\ud}{\,\mathrm{d}}
\newcommand{\CubeHomL}{Y}

\DeclareMathOperator{\Lapl}{\triangle}

\title{Homogenization at different linear scales, bounded
  martingales and the Two-Scale Shuffle limit}
\author{Kévin Santugini\thanks{
    Univ. Bordeaux, IMB, UMR 5251, F-33400 Talence, France. / 
    CNRS, IMB, UMR 5251, F-33400 Talence, France. /
    INRIA, F-33400 Talence, France. /
    \texttt{Kevin.Santugini@math.u-bordeaux1.fr}}
}
\date{\today}

\begin{document}
\maketitle
\begin{abstract}
In this paper, we consider two-scale limits
obtained with increasing homogenization periods, each period being an entire multiple of the previous
one. We establish that, up to a measure preserving rearrangement, these two-scale limits form a
martingale which is bounded: the rearranged two-scale limits
themselves converge both strongly in $\mathrm{L}^2$ and almost everywhere when the period tends to
$+\infty$. This limit, called the Two-Scale Shuffle limit, contains all the information present in all the
two-scale limits in the sequence.
\end{abstract}

\section{Introduction}
Homogenization is used to study the solutions to
equations when there are multiple scales of interest, usually a
microscopic one and a macroscopic one. In particular, one may consider the solutions
$u_\varepsilon$ to a partial differential equation with locally $\varepsilon$-periodic
coefficients and study their behavior as
the small period $\varepsilon$ tends to $0$. Two-scale convergence,
introduced by G.~Nguetseng~\cite{Nguetseng:TheoryHomogenization} and
G.~Allaire~\cite{Allaire:TwoScaleConvergence}, is
suited to study this particular subset of homogenization problems
called periodic homogenization. It was later extended to the case of
periodic surfaces by M.~Neuss Radu~\cite{NeussRadu:These,
  NeussRadu:ExtensionsTwoScaleConv}
and G.~Allaire, A.~Damlamian and
U.~Hornung~\cite{Allaire.Damlamian.Hornung:TwoScaleSurfacicConvergence}.
It can also be used in the
presence of periodic holes in the geometry, 
see~\cite{Cioranescu.Paulin:HomogenizationOpenSetHoles,Damlamian.Donato:WhichSequencesHolesHomogenization}
or to homogenize multilayers~\cite{Santugini:HomogenizationMultilayerLL,Santugini:HomogenizationMultilayerHeat}.

Intuitively, two-scale convergence introduces the concept of two-scale limit $u_0$ which is a function of both a
macroscopic variable $\bm{x}$ ---also called slow variable---  and a microscopic $p$-periodic variable
$\bm{y}$ ---also called fast variable--- such that, in some ``meaning'', $\bm{x}\mapsto
u_0(\bm{x},\bm{x}/\varepsilon)$ is a good approximation of $u_\varepsilon$.

As indicated by its name, two-scale convergence captures the behavior
at two scales:
the macroscopic one and the $p\varepsilon$-periodic one. However,
two-scale convergence does not capture all phenomena that happens at
a scale linear in $\varepsilon$ but only those whose length
scale is $p\varepsilon/m$ where $m$ is an integer. The two-scale
limit of a sequence 
depends not only on the asymptotic scale, but also on the precise value
of the chosen period. For example, any
phenomena happening at the length scale of $2\varepsilon$ will not be
fully apparent in the two-scale limit computed with period $\varepsilon$. 
The two-scale limit computed with period $2\varepsilon$ will contain
no less ---and might actually  contain more--- information than the
two-scale limit computed with period $\varepsilon$. For example, the
homogenization of $\sin(2\pi x/\varepsilon)+\sin(\pi x/\varepsilon)$ gives
a two-scale limit of $u_0:(x,y)\mapsto\sin(2\pi y)$ if computed with
the homogenization period $\varepsilon$, \textit{i.e.}, when $p=1$, and $u_0:(x,y)\mapsto\sin(2\pi
y)+\sin(\pi y)$ if computed with the 
homogenization period $2\varepsilon$, \textit{i.e.}, when $p=2$.
Furthermore, if we choose $p=1/2$, then the two-scale limit is none
other than the null function. Worse, the scale factor $p$ could be irrational.

The choice of the scale factor $p$ used in the homogenization
process is therefore of utmost importance in two-scale
convergence. Using a badly chosen scale
factor $p$ may and will often cause a huge loss of information. At worst, we recover
no more information than the one obtained by the standard weak
$\mathrm{L}^2$ limit: if
$p\varepsilon$ is the correct choice of homogenization period, the two-scale limit 
computed with period $\lambda p\varepsilon$ where $\lambda$ is an irrational number
should, intuitively, carry no information about what happens at scale
$p\varepsilon$.

Fortunately, there is usually a natural choice of period: the coefficients of
the partial differential equation are often chosen locally
$\varepsilon$-periodic. The most
natural choice is to choose $p=1$, \textit{i.e.}, to consider the
correct microscopic scale for $u_\varepsilon$ is $\varepsilon$ itself.
If there are two important periods to consider $p\varepsilon$
and $p'\varepsilon$, the intuitive solution is to choose a period that
is an entire multiple of both. However, this can only be done if
  the ratio $p/p'$
between the two scale factors
is a rational number. 

When the two-scale limit depends on the fast variable, 
we may 
consider an homogenization period of $p_2\varepsilon$ instead of
$p_1\varepsilon$ where $p_2/p_1$ is a positive integer. The two-scale limit
computed with the homogenization period $p_2\varepsilon$ contains more
information than the two-scale limit computed with the homogenization
period $p_1\varepsilon$. It is then natural to study the behavior of the
two-scale limit as the scale factor tends to
$+\infty$. G.~Allaire and C.~Conca studied in~\cite{Allaire.Conca:BlochWave}
a similar problem and  established, for an elliptic problem, the behavior of the
spectra of the equation satisfied by the two-scale limit as the scale
factor $p$ goes to $+\infty$. G.~Ben Arous and
H.~Owhadi~\cite{Arous.Owhadi:MultiscaleHomogenization}
studied the behavior of the Brownian motion in a periodic
potential using multiscale homogenization when the ratio between two
successive scales is bounded from above and below.

In this paper, we consider various two-scale limits, each computed
with a different homogenization period. In particular, we consider a sequence of
periods $(p_n)_{n\in\mathbb{N}}$ such that for all integers $n$, $p_{n+1}/p_{n}$ is a
positive integer and we study the two-scale limit of
$(u_\varepsilon)_{\varepsilon>0}$ computed with the homogenization
period $p_n\varepsilon$. This two-scale limit, denoted $u_{0,p_n}$, is $p_n$-periodic in each
component of its fast
variable. Since
$p_{n+1}$ is always an entire multiple of $p_n$, 
one can always recover the two-scale limit $u_{0,p_{n}}$ from the two
scale limit $u_{0,p_{n+1}}$.  
If $p_{n+1}=m_np_n$ and in dimension $d\geq1$:
\begin{equation*}
u_{0,p_{n}}(\bm{x},\bm{y})=\frac{1}{m_n^d}\sum_{\bm{\alpha}\in\llbracket0,m_n-1\rrbracket^d}
	u_{0,p_{n+1}}(\bm{x},\bm{y}+p_n\bm{\alpha}).
\end{equation*}

The sequence of two-scale limits $(u_{0,p_n})_{n\in\mathbb{N}}$
yields increasing information on the asymptotic behavior of $(u_\varepsilon)_{\varepsilon>0}$.
A natural question is whether
the two-scale limits $u_{0,p_n}$ themselves converge whenever 
$n$ tends to $+\infty$.  \textit{I.E.}, does there exist a function
that carry the information of all the $p_n$-two-scale limits?
The goal of our paper is to answer this question. The answer is
positive. We show in this
paper that the sequence of two-scale limits is, after a measure
preserving rearrangement,
a bounded martingale in $\mathrm{L}^2$ and therefore converges both strongly in
$\mathrm{L}^2$ and almost everywhere to a function we call the Two-Scale Shuffle limit.

In \S\ref{sect:PreviouslyKnownResults}, we remind the reader of 
previously known results: two-scale
convergence 
and the convergence properties
of bounded martingales.
In \S\ref{sect:rearrangement}, we show how 
the different two-scale limits are related to each other through
martingale-like equalities and
explain how to transform these two-scale limits to get a genuine martingale.
This leads to our stating of our main theorem:
Theorem~\ref{theo:TwoScaleLimitsMartingale}
in which we show that in a certain meaning the two-scale limits
themselves converge to the Two-Scale Shuffle limit. In addition, we
also state in Corollary~\ref{corrol:RecoveryTwoScale} that
all the information present in all the two-scale limits is contained
in the Two-Scale Shuffle limit.
In \S\ref{sect:application}, we use this result on the
heat equation in multilayers with transmission conditions between
adjacent layers and establish, for this particular example, the equation satisfied by the Two-Scale
Shuffle limit in Theorem~\ref{theo:ShuffleLimitHeatEquationMultilayers}.

\section{Notations, prerequisites and known
    results}\label{sect:PreviouslyKnownResults}
Throughout this paper, if $x$ is in $\mathbb{R}$, we denote by $\lfloor x\rfloor$ the
  integer part of $x$. We also denote by $\llbracket n_1,n_2\rrbracket$ the set $\lbrack n_1,
  n_2\rbrack\cap\mathbb{N}$. To make the present paper as 
  self-contained as possible, we recall in this section
  known results on the two main mathematical tools we use to prove our
  main theorem: two-scale convergence in
  \S\ref{subsect:TwoScaleConvergence}, and classical results on
  the convergence of bounded martingales in \S\ref{subsect:ProbabilityTheory}.

\subsection{The classical notion of two-scale convergence}\label{subsect:TwoScaleConvergence}
First, as in~\cite{Allaire:TwoScaleConvergence}, we introduce some notations. In this paper, $p$ always
refer to a scale factor. It remains constant while taking the two-scale
limit. However, the goal of this paper is to observe the behavior of
the two-scale limits as $p$ tends to $+\infty$.

By $\Omega$, we denote a bounded open domain of $\mathbb{R}^d$ where $d\geq 1$. By $\CubeHomL_p$,
we denote the cube $\lbrack0,p\rbrack^d$. 
By $\mathrm{L}^2_\#(\CubeHomL_p)$, we denote the space of measurable functions 
defined over $\mathbb{R}^d$, that are $p$-periodic in each variable
and that are square integrable over $\CubeHomL_p$.
By $\mathcal{C}_\#(\CubeHomL_p)$, we denote the
set of continuous functions defined on $\mathbb{R}^d$ that are $p$-periodic
in each variable.

We reproduce the now classical definition of two-scale convergence
found in~\cite{Allaire:TwoScaleConvergence,Nguetseng:TheoryHomogenization}. For convenience, we added the scale factor $p$.
\begin{definition}[Two-scale convergence]\label{defin:TwoScaleConvergence}
Let $p$ be a positive real. 
A sequence $(u_\varepsilon)_{\varepsilon>0}$ belonging to
$\mathrm{L}^2(\Omega)$ is said to $p$-two-scale converge
if there exists $u_{0,p}$ in $\mathrm{L}^2(\Omega\cart\CubeHomL_p)$
such that:
\begin{equation}
\lim\limits_{\varepsilon\to0}
\int_\Omega u_{\varepsilon}(\bm{x})
	\psi\left(\bm{x},\frac{\bm{x}}{\varepsilon}\right)
	\ud\bm{x}
=\frac{1}{p^d}\int_\Omega\int_{\CubeHomL_p}
	u_{0,p}(\bm{x},\bm{y})\psi\left(\bm{x},\bm{y}\right)\ud\bm{y}\ud\bm{x},
\end{equation}
for all $\psi$ in $\mathrm{L}^2(\Omega;\mathcal{C}_\#(\CubeHomL_p))$.
\end{definition}
It is a common abuse of notation to also designate by
  $u_{0,p}$ the unique extension of  $u_{0,p}$ to
  $\Omega\cart\mathbb{R}^d$ that is
$p$-periodic in the last $d$ variables.

G.~Allaire, see~\cite{Allaire:TwoScaleConvergence},  and G.Nguetseng,
see~\cite{Nguetseng:TheoryHomogenization}, proved
that any sequence of functions bounded in $\mathrm{L}^2$ has a 
subsequence that two-scale converges. Let's
reproduce this precise compactness result.
\begin{theorem}\label{theo:TwoScaleConvergenceCompactness}
Let $(u_\varepsilon)_{\varepsilon>0}$ be a sequence of functions
bounded in $\mathrm{L}^2(\Omega)$. Then, there exist
$u_{0,p}$ in $\mathrm{L}^2(\Omega\cart\rbrack0,p\lbrack^d)$
and a subsequence $\varepsilon_k$ converging to $0$ such that 
\begin{equation}\label{eq:defDoubleConv}
\lim\limits_{k\to\infty}
\int_\Omega u_{\varepsilon_k}(\bm{x})
	\psi\left(\bm{x},\frac{\bm{x}}{\varepsilon_k}\right)\ud\bm{x}
=\frac{1}{p^d}\int_\Omega\int_{\CubeHomL_p}
	u_{0,p}(\bm{x},\bm{y})\psi\left(\bm{x},\bm{y}\right)
	\ud\bm{y}\ud\bm{x},
\end{equation}
for all $\psi$ in $\mathrm{L}^2(\Omega;\mathcal{C}_\#(\CubeHomL_p))$.
\end{theorem}
\begin{proof}
See G.~Allaire~\cite[Theorem~1.2]{Allaire:TwoScaleConvergence} and
G. Nguetseng~\cite[Theorem~2]{Nguetseng:TheoryHomogenization}.
The presence of the scale factor $p$ has no impact on the proof.
\end{proof}

We also have the classical proposition
\begin{proposition}\label{prop:TwoScaleLimitL2Bound}
Let  $u_\varepsilon$ $p$-two-scale converges to $u_{0,p}$. Then,
\begin{equation*}
\frac{1}{p^{d/2}}\lVert u_{0,p}\rVert_{\mathrm{L}^2(\Omega\cart\CubeHomL_p)}\leq 
\liminf\limits_{\varepsilon\to0}
	\lVert u_\varepsilon\rVert_{\mathrm{L}^2(\Omega)}.
\end{equation*}
\end{proposition}
\begin{proof}
See G.~Allaire~\cite[Proposition~1.6]{Allaire:TwoScaleConvergence}. The presence of the scale factor $p$ has no impact on the proof.
\end{proof}

The next proposition is easy to derive from Theorem~\ref{theo:TwoScaleConvergenceCompactness}.
\begin{proposition}\label{prop:TwoScaleConvergenceMultipleScaleFactors}
Let $(p_n)_{n\in\mathbb{N}}$ be an increasing sequence of positive real numbers.
Let $(u_\varepsilon)_{\varepsilon>0}$ be a sequence of functions
bounded in $\mathrm{L}^2(\Omega)$. Then, there exist
a subsequence $(\varepsilon_k)_{k\in\mathbb{N}}$ converging to $0$, and a 
sequence of functions $u_{0,p_n}$ in
  $\mathrm{L}^2(\Omega\cart\rbrack0,p_n\lbrack^d)$
such that, for any non-negative  integer $n$, 
the sequence $(u_{\varepsilon_k})_{k\in\mathbb{N}}$ $p_n$-two-scale converges
to $u_{0,p_n}$.
\textit{I.E.}, such that for all integers $n$:
\begin{equation*}
\lim\limits_{k\to\infty}
\int_\Omega u_{\varepsilon_k}(\bm{x})
	\psi\left(\bm{x},\frac{\bm{x}}{\varepsilon_k}\right)\ud\bm{x}
=\frac{1}{p_n^d}\int_\Omega\int_{\CubeHomL_{p_n}}
	u_{0,p_n}(\bm{x},\bm{y})\psi\left(\bm{x},\bm{y}\right)
	\ud\bm{y}\ud\bm{x},
\end{equation*}
for all $\psi$ in $\mathrm{L}^2(\Omega;\mathcal{C}_\#(\CubeHomL_{p_n}))$.
\end{proposition}
\begin{proof}
Apply Theorem~\ref{theo:TwoScaleConvergenceCompactness}
multiple times and proceed via diagonal extraction.
\end{proof}

Our goal in this paper is to study the limit of $u_{0,p_n}$ as $p_n$
tends to $+\infty$. 

\subsection{Convergence of bounded martingales}\label{subsect:ProbabilityTheory}

In this section, we recall the
notions of probability theory needed to prove our main theorem.
In particular, we are interested in using the convergence properties 
of bounded martingales. For more details, the reader may 
consult~\cite{Kallenberg:FoundationsProbability}.
We assume the reader to be familiar with the notions of $\sigma$-field
and $\sigma$-additivity in measure theory.

We use the following common notations:
\begin{itemize}
\item If $C$ is a subset of $\mathcal{P}(X)$,  we denote by $\sigma(C)$
the smallest $\sigma$-field in $X$ that contains $C$.
\item If $D$ is a topological space, we denote by $\mathcal{B}(D)$
  the set of all Borel sets in $D$, \textit{i.e.}, the smallest $\sigma$-field
  containing all the open subsets of $D$.
\end{itemize}

\begin{definition}[Measurable space]
A pair $(X, \mathcal{F})$ is said to be a
measurable space if $\mathcal{F}$ is a $\sigma$-field in $X$.
\end{definition}

\begin{definition}[Measure space]
A triplet $(X, \mathcal{F},\mu)$ is said to be a
measure space if $(X, \mathcal{F})$ is a measurable space and if
$\mu$ is a positive $\sigma$-additive measure on $(X,
\mathcal{F})$.
\end{definition}

A measure space $(X, \mathcal{F}, \mu)$ is said to be finite if
$\mu(X)<+\infty$.
A measure space $(X, \mathcal{F}, \mu)$ is said to be $\sigma$-finite if
$X$ is the countable union of $\mathcal{F}$-measurable sets 
of finite measure.
A measure space $(X, \mathcal{F}, \mathbb{P})$ is said to be a probability space
if $\mathbb{P}(X)=1$.

We start by recalling the definition of
conditional expectation,
see~\cite[ch.~6, Theorem~6.1]{Kallenberg:FoundationsProbability} for more
details. Usually, the conditional 
expectation is defined for probability
spaces. The definition extends without problem to finite measure spaces
and even, to some extent, to $\sigma$-finite measure spaces.
\begin{definition}[Conditional expectation]\label{defin:ConditionalExpectation}
Let $(X,\mathcal{F},\mu)$ be a measure  space with $\mu$ being
positive and  $\sigma$-additive. Let $\mathcal{G}$ be a $\sigma$-field
such that 
$\mathcal{G}\subset\mathcal{F}$ and $(X,\mathcal{G},\mu)$ is
also $\sigma$-finite. Let $f:X\to\mathbb{R}$ be
$\mathcal{F}$-measurable and in $L^1_{\mathrm{loc}}(X,\mu)$. 
The conditional expectation of $f$ with respect to the $\sigma$-field
$\mathcal{G}$ is denoted by $\mathbb{E}(f\vert\mathcal{F})$, and
is defined as the unique, up to a modification on a set of null measure,
$\mathcal{G}$-measurable function $g$ such that
\begin{equation*}
\int_Bg(\omega)\ud\omega=\int_Bf(\omega)\ud\omega,
\end{equation*}
for all $B$ in $\mathcal{G}$.
\end{definition}
The existence of the conditional expectation is given by
Radon-Nikodym theorem. The measure $\mu$ need not be a probability
measure. However, to apply Radon-Nykodim theorem,
$(X,\mathcal{G},\mu)$ needs to be $\sigma$-finite, hence the
restriction in the definition. A statement and a proof of the
Radon-Nikodym theorem can be found 
in~\cite[Theorem~6.10]{Rudin:1987:RealAndComplexAnalysis3rdEd}.

It is not enough that $(X,\mathcal{F},\mu)$ be $\sigma$-finite in Definition~\ref{defin:ConditionalExpectation}.
\begin{remark}\label{rem:TrapSigmaFiniteMeasure}
When $\mathcal{G}\subset\mathcal{F}$, it does not follow from $(X,\mathcal{F},\mu)$ being $\sigma$-finite 
that $(X,\mathcal{G},\mu)$ is also $\sigma$-finite.
A counter-example is easily
obtained by setting $\mathcal{G}:=\{\emptyset,X\}$
whenever $\mu(X)=+\infty$. 
\end{remark}
In our main theorem, we restrict ourselves to the case of finite
measures. However, Remark~\ref{rem:TrapSigmaFiniteMeasure} will explain why the martingale
approach doesn't quite work for the most natural attempt to define 
a convergence for two-scale limits, see \S\ref{subsect:MartingaleEquality}.

In order to define martingales, we 
remind the reader of the definition of filtration.
We limit ourselves to filtrations indexed by the set $\mathbb{N}$.
See~\cite[ch.~7, p.~120]{Kallenberg:FoundationsProbability} for more details.
\begin{definition}[Filtrations]
Let $(X,\mathcal{F})$ be a measurable space. A sequence $(\mathcal{F}_n)_{n\in\mathbb{N}}$
of $\sigma$-fields, $\mathcal{F}_n\subset \mathcal{F}$ is a filtration if,
for all non-negative integers $n$, 
$\mathcal{F}_n$ is a subset of $\mathcal{F}_{n+1}$.
\end{definition}

We now recall the definition of martingales.
\begin{definition}[Martingales]\label{defin:martingales}
Let $(X,\mathcal{F},\mu)$ be a $\sigma$-finite measure space.
Let $(\mathcal{F}_n)_{n\in\mathbb{N}}$ be a filtration on 
$(X,\mathcal{F},\mu)$ such that $(X,\mathcal{F}_0,\mu)$ is $\sigma$-finite.

A sequence $(f_n)_{n\in\mathbb{N}}$ is said to be a
$(\mathcal{F}_n)_{n\in\mathbb{N}}$-martingale, if for all non-negative
integers $n$ and $j$,
\begin{equation*}
f_{n}=\mathbb{E}(f_{n+j}\vert\mathcal{F}_n).
\end{equation*}
\textit{I.E.}, if $f_n$ is $\mathcal{F}_n$-measurable and 
 if for all $F$ in $\mathcal{F}_n$:
\begin{equation}\label{eq:DefMartingaleEquality}
\int_Ff_{n}(\omega)\ud\omega=
\int_Ff_{n+j}(\omega)\ud\omega
\end{equation}
\end{definition}

We now reproduce the convergence results of bounded martingales:
\begin{theorem}[Convergence of bounded martingales]\label{theo:ConvergenceBoundedMartingales}
Let $(X,\mathcal{F},\mu)$ be a measure space with finite measure. 
Let $(\mathcal{F}_n)_{n\in\mathbb{N}}$ be a filtration on the
measurable space $(X,\mathcal{F})$. Let $q$ be in $\rbrack1,+\infty\lbrack$.
Let $(f_n)_{n\in\mathbb{N}}$ be a
$(\mathcal{F}_n)_{n\in\mathbb{N}}$-martingale such that the sequence
$(f_n) _{n\in\mathbb{N}}$ is bounded
in $L^q(X)$. Then, the sequence $(f_n)_{n\in\mathbb{N}}$ converges both
almost everywhere and strongly in $L^q(X,\mathbb{P})$.
\end{theorem}
\begin{proof}
See~\cite[Corollary~7.22] {Kallenberg:FoundationsProbability} for the
strong $L^q$ convergence. The almost everywhere convergence is stated 
in~\cite[Theorem~7.18] {Kallenberg:FoundationsProbability}
and holds even for $q=1$. While these two results are stated for
probability measures, the finite measures case is easily deduced from 
the probability measure case by considering the probability measure 
$\mu(\cdot)/\mu(X)$.
\end{proof}

The above theorem extends, at least partially, to $\sigma$-finite measures:
\begin{remark}
In Theorem~\ref{theo:ConvergenceBoundedMartingales}, if the
probability space
$(X,\mathcal{F},\mathbb{P})$ is replaced with $\sigma$-finite
measure space $(X,\mathcal{F},\mu)$  such that $(X,\mathcal{F}_0,\mu)$
is also $\sigma$-finite,  then the bounded martingales
converge almost everywhere and at least in $L^q_{loc}$. It is
unknown to the author if the strong $L^q$ convergence can be generalized to
the $\sigma$-finite  case.
\end{remark}

\section{Two-scale limits and bounded martingales}\label{sect:rearrangement}
In this section, we always assume both of the following
  assumptions are satisfied:
\begin{assumption}[Integer scale ratios]\label{assum:IntergerScaleFactors}
We are given a real sequence $(p_n)_{n\in\mathbb{N}}$, such that for
all $n$ in $\mathbb{N}$, $p_n>0$  and $p_{n+1}$ is an entire
multiple of $p_n$. 
Moreover, we set for $n\geq1$, $m_n:=p_{n}/p_{n-1}\in\mathbb{N}$, and for $n\geq0$,
$M_n:=p_n/p_0\in\mathbb{N}$. 
\end{assumption}
\begin{assumption}\label{assum:BoundedL2Sequence}
We are given a sequence of functions
$(u_\varepsilon)_{\varepsilon>0}$ bounded in $\mathrm{L}^2(\Omega)$ and 
 a decreasing sequence of positive
$(\varepsilon_k)_{k\in\mathbb{N}}$ such that the sequence
$(u_{\varepsilon_k})_{k\in\mathbb{N}}$ $p_n$-two-scale converges for
all integers $n$ to a function $u_{0,p_n}$ that belongs to $\mathrm{L}^2(\Omega\cart\rbrack0,p_n\lbrack^d)$. 
\end{assumption}
This last assumption is justified by
Proposition~\ref{prop:TwoScaleConvergenceMultipleScaleFactors}.


Our goal is to study the convergence of the two-scale limits $u_{0,p_n}$
when $n$ goes to infinity. In this section, we proceed as follows:
we begin by establishing a useful equality 
that looks like a martingale equality
in~\S\ref{subsect:MartingaleEquality}, then we propose a rearrangement
of the two-scale limits in~\S\ref{subsect:RearrangementSequence}, and
finally propose another rearrangement of the two-scale limits 
in~\S\ref{subsect:RearrangementContinuous}, the shuffle, which transform the sequence of
two-scale limits into a bounded martingale.

\subsection{An almost martingale equality}\label{subsect:MartingaleEquality}
We start with a simple but essential proposition.
\begin{proposition}\label{prop:MartingaleHomog}
  Suppose both assumptions~\ref{assum:IntergerScaleFactors} 
  and~\ref{assum:BoundedL2Sequence} are satisfied.
  Then, for all $j$ in $\mathbb{N}$, all
  $n$ in $\mathbb{N}$, almost all $\bm{x}$ in $\Omega$ and almost all $\bm{y}$
  in $\CubeHomL_p$:
  \begin{equation}\label{eq:MartingaleHomog}
    u_{0,p_{n}} (\bm{x},\bm{y})=
    \left(\frac{p_n}{p_{n+j}}\right)^d\sum_{\bm{\alpha}\in\llbracket0,p_{n+j}/p_n-1\rrbracket^d}
    u_{0,p_{n+j}}(\bm{x},\bm{y}+\bm{\alpha} p_n).
  \end{equation}
\end{proposition}
\begin{proof}
Let $\phi$ belong to
$\mathcal{C}^\infty(\overline{\Omega}\cart\mathbb{R}^d)$ be
$p_n$-periodic in the last $d$ variables. Since $p_{n+j}/p_n$ is an integer,
$\phi$ is also $p_{n+j}$-periodic in the last $d$ variables.
We take the limit of
$\int_{\Omega}u_\varepsilon(\bm{x})\phi(\bm{x},\bm{x}/\varepsilon)\ud\bm{x}$, 
as $\varepsilon$ tends to $0$, in the sense of two-scale convergence
for both scale factors $p_{n+j}$ and $p_n$:
\begin{multline*}
 \frac{1}{p_n^d}
\int_\Omega\int_{\CubeHomL_{p_n}}u_{0,p_n}(\bm{x},\bm{y})\phi(\bm{x},\bm{y})\ud\bm{y}\ud\bm{x}
 =\\=
 \frac{1}{p_{n+j}^d}
 \int_\Omega\int_{\CubeHomL_{p_{n+j}}}
u_{0,{p_{n+j}}}(\bm{x},\bm{y})
\phi(\bm{x},\bm{y})\ud\bm{y}\ud\bm{x},
 =\\=
 \frac{1}{p_{n+j}^d}
 \int_\Omega\int_{\CubeHomL_{p_n}}
 \left(\vphantom{\sum}\smash{\sum_{\bm{\alpha}\in(\llbracket0,p_{n+j}/p_n-1\rrbracket^d}}u_{0,p_{n+j}}(\bm{x},\bm{y}+\bm{\alpha}p)\right)
 \phi(\bm{x},\bm{y})\ud\bm{y}\ud\bm{x}.\qedhere
\end{multline*}
\end{proof} 

The most natural approach is to consider the $u_{0,p_n}$ as
functions defined over $\Omega\cart\mathbb{R}^d$ and to study their
convergence in some meaning in $\Omega\cart\mathbb{R}^d$.  
Such a convergence result would be ideal as the
intuitive meaning of the limit would be easy to grasp. 
Equality~\eqref{eq:MartingaleHomog} is similar to the martingale
defining equality~\eqref{eq:DefMartingaleEquality}. Would it be
possible to use the classical convergence properties
of martingales, see Theorem~\ref{theo:ConvergenceBoundedMartingales},
to prove the existence of a limit to the $u_{0,p_n}$? Unfortunately,
the martingale approach doesn't work in this setting but 
the attempt is, nevertheless, instructive. First, we try to construct
a filtration $(\mathcal{F}_n)_{n\in\mathbb{N}}$ for the $u_{0,p_n}$. For all positive integer
$n$, the $u_{0,p_n}$ are $p_n$-periodic with respect to the last $d$
variables. Let $\mathcal{F}_n$ be the set  of all Borel subsets of
$\Omega\cart\mathbb{R}^d$ that are invariant by translation of $\pm
p_n$ along any of the last $d$ directions of
$\Omega\cart\mathbb{R}^d$. 
Clearly, $u_{0,p_n}$ is $\mathcal{F}_n$-measurable. However, the
measure space $(\Omega\cart\mathbb{R}^d,\mathcal{F}_n,\mu)$ where
$\mu$ is the Lebesgue measure is not $\sigma$-finite: any
$\mathcal{F}_n$ measurable subset of $\Omega\cart\mathbb{R}^d$ is
either of null measure or of infinite measure. Therefore, the concept
of $(\mathcal{F}_n)_{n\in\mathbb{N}}$-martingale is ill-defined, see
Definition~\ref{defin:martingales}
and Remark~\ref{rem:TrapSigmaFiniteMeasure}. Should we attempt to
verify whether the martingale defining
equality~\eqref{eq:DefMartingaleEquality}  hold, we would get
either $+\infty$ or $0$ on both sides of the equation.


However, the martingale defining equality~\eqref{eq:DefMartingaleEquality} is
satisfied if one replaces the Lebesgue integral of $\mathbb{R}^d$ by the limit of the
mean over a ball as its radius tends to $+\infty$. \textit{I.E.}, we have for all $F$ in $\mathcal{F}_n$
\begin{multline*}
\lim_{R\to+\infty}\frac{1}{\lvert
 B(\bm{0},R)\rvert}\int_\Omega\int_{B(\bm{0},R)}
\mathds{1}\{(\bm{x},\bm{y})\in F\}u_{0,p_n}(\bm{x},\bm{y})\ud\bm{y}\ud\bm{x}
=\\=
\lim_{R\to+\infty}\frac{1}{\lvert
 B(\bm{0},R)\rvert}\int_\Omega\int_{B(\bm{0},R)}
\mathds{1}\{(\bm{x},\bm{y})\in F\}u_{0,p_{n+j}}(\bm{x},\bm{y})\ud\bm{y}\ud\bm{x},
\end{multline*}
where $B(\bm{0},R)$ is the open ball of $\mathbb{R}^{d}$ centered on $0$
and of radius $R$ and where $\lvert A\rvert$ is the Lebesgue measure of set $A$.
Unfortunately, we were unable to derive a direct convergence result
using this pseudo-martingale equality. To proceed further, we need to
transform the two-scale limits $u_{0,p_n}$ in order to get genuine martingales. 

\subsection{Rearrangement of the two-scale
   limits with integers}\label{subsect:RearrangementSequence}
In the previous section, we established a ``martingale-like'' equality 
for the two-scale limits $u_{0,p_n}$. To get genuine martingales in
the sense of Definition~\ref{defin:martingales}, we need to rearrange
the $u_{0,p_n}$. While we are unable to prove a convergence 
for the rearrangement of the two-scale limits presented in this
section, the ideas behind this rearrangement
provide insight on the next section where we introduce another
rearrangement and prove its convergence.

In this section, we rearrange the $u_{0,p_n}$ by
introducing a new variable $\bm{\alpha}$ that belongs to
$\mathbb{Z}^d$. The rearrangement, denoted by $v_{p_n}$, depends on the
slow variable $\bm{x}\in \Omega$, on a fast variable
$\bm{y}\in\lbrack0,p_0\lbrack^d$,
and on the new variable $\bm{\alpha}$.  To rearrange the $u_{0,p_n}$
into the $v_{p_n}$, we subdivide $\Omega\cart\lbrack0,p_n\lbrack^d$
 into $M_n^d=(p_n/p_0)^d$ sets
$\Omega\cart\prod_{i=1}^d\lbrack\alpha_i,\alpha_i+p_0\lbrack$.
Each of these sets is the product of $\Omega$ with an hypercube 
indexed by $\bm{\alpha}=(\alpha_1,\ldots,\alpha_d)$ and we define
$v_{p_n}(\cdot,\bm{\alpha},\cdot)$ as taking in
$\Omega\cart\lbrack0,p_0\lbrack^d$ 
the same values $u_{0,p_n}$ does in 
$\Omega\cart\prod_{i=1}^d\lbrack\alpha_i,\alpha_i+p_0\lbrack$. The
variable $\bm{y}$ represents the position of the fast variable inside
each hypercube. \textit{I.E.}, we set:
\begin{equation*}
\begin{split}
v_{p_n}:\Omega\cart \mathbb{Z}^d\cart\CubeHomL_{p_0}&\to\mathbb{R},\\
(\bm{x},\bm{\alpha},\bm{y})&\mapsto u_{0,p_n}(\bm{x},\bm{y}+p_0\bm{\alpha}).
\end{split}
\end{equation*}



We have the following proposition
\begin{proposition}\label{prop:RearrangementIsImbricated}
For all $n$ in $\mathbb{N}$, for almost all $\bm{x}$ in $\Omega$ and $\bm{y}$ in
$\CubeHomL_{p_0}$, the $\bm{\alpha}$-indexed sequence
$(v_{p_n}(\bm{x},\bm{\alpha},\bm{y}))_{\bm{\alpha}\in\mathbb{Z}^d}$ is
$M_n$-periodic in each direction  of $\bm{\alpha}$. Moreover:
\begin{equation}\label{eq:defimbricatedperiodicsequences}
v_{p_n}(\bm{x},\bm{\alpha},\bm{y})=\left(\frac{M_n}{M_{n+j}}\right)^d\sum_{\bm{\beta}\in\llbracket0,M_{n+j}/M_n-1\rrbracket^d}v_{p_{n+j}}(\bm{x},\bm{\alpha}+M_n\bm{\beta},\bm{y}),
\end{equation}
for all $\bm{\alpha}$ in $\mathbb{Z}^d$.
\end{proposition}
\begin{proof}
This is a direct consequence of Proposition~\ref{prop:MartingaleHomog}.
\end{proof}

This in turn should encourage us to look at the following problem.
\begin{problem}\label{problem:ImbricatedPeriodicConvergence}
Let's call ``imbricated $(M_n)_n$-periodic $d$-dimensional
sequences'', sequences that satisfy the following properties
$(t_{n,\bm{\alpha}})_{n\in\mathbb{N},\bm{\alpha}\in\mathbb{Z}^d}$ such that
\begin{itemize}
\item for all $n$ in $\mathbb{N}$, the $\bm{\alpha}$-indexed 
sequence $(t_{n,\bm{\alpha}})_{\bm{\alpha}\in\mathbb{Z}^d}$
is $M_n$-periodic in each direction of $\bm{\alpha}$, \textit{i.e.}
such that for all $n$ in $\mathbb{N}$, 
for all $\bm{\alpha}$ in $\mathbb{Z}^d$,
and for all $\bm{\beta}$ in $\mathbb{Z}^d$:
\begin{equation*}
t_{n,\bm{\alpha}}=t_{n,\bm{\alpha}+M_n\bm{\beta}},
\end{equation*}
\item for all $n$ in $\mathbb{N}$, and for all 
$\bm{\alpha}$ in $\mathbb{Z}^d$,
\begin{equation*}
t_{n,\bm{\alpha}}=\left(\frac{M_n}{M_{n+j}}\right)^d\sum_{\bm{\beta}\in\llbracket0,M_{n+j}/{M_n}-1\rrbracket^d}t_{{n+j},\bm{\alpha}+M_n\bm{\beta}}.
\end{equation*}
\end{itemize}
Study the convergence of $(t_{n,\bm{\alpha}})_{n\in\mathbb{N},\bm{\alpha}\in\mathbb{Z}^d}$ as $n$ tends to
$+\infty$. Under which condition does
there exist a sequence $t_{\infty,\bm{\alpha}}$ such that for all
non-negative integers $n$
\begin{equation*}
t_{n,\bm{\alpha}}=\lim_{N\to+\infty}\frac{1}{N^d}\sum_{\bm{\beta}\in\llbracket0,N-1\rrbracket^d}t_{\infty,\bm{\alpha}+M_n\bm{\beta}}.
\end{equation*}
or such that
\begin{equation*}
t_{n,\bm{\alpha}}=\lim_{N\to+\infty}\frac{1}{2^dN^d}\sum_{\bm{\beta}\in\llbracket-N,N-1\rrbracket^d}t_{\infty,\bm{\alpha}+M_n\bm{\beta}}.
\end{equation*}
or both?
\end{problem}
By Proposition~\ref{prop:RearrangementIsImbricated}, for almost all $\bm{x}$ in $\Omega$ and $\bm{y}$ in
$\CubeHomL_{p_0}$, the $\bm{\alpha}$-indexed sequences
$(v_{p_n}(\bm{x},\bm{\alpha},\bm{y}))_{\bm{\alpha}\in\mathbb{Z}^d}$ are imbricated
$(M_n)_n$-periodic $d$-dimensional sequences. Solving Problem~\ref{problem:ImbricatedPeriodicConvergence}
would be the first step in having a very elegant limit to the $v_{p_n}$ as a
function defined on $\Omega\cart\mathbb{Z}^d\cart\CubeHomL_{p_0}$.
Unfortunately, we do not have an answer for Problem~\ref{problem:ImbricatedPeriodicConvergence}. While this
sequence is morally a martingale with respect to the filtration made
of the $\sigma$-fields
$\{\bm{\alpha}+M_n\mathbb{Z}^d,\bm{\alpha}\in\llbracket0,M_n-1\rrbracket^d\}$,
it technically is not: we have the same problem we had in the previous
section. To conclude with bounded martingales on the convergence, we would
need a measure $\mu$ on $\mathbb{Z}^d$ such that
$\mu(\mathbb{Z}^d)=1$, invariant by translation  and such that $\mu(m\mathbb{Z}^d)=1/m$ whenever $m$ is
an integer different from $0$. Such a measure cannot be
$\sigma$-additive. If we remove the $\sigma$ additivity constraint,
then $\mu$ exists: just set
\begin{equation*}
\mu(A):=\lim_{N\to+\infty}\frac{\#(A\cap\llbracket-N,N\rrbracket^d)}{(2N+1)^d}.
\end{equation*}
It is unknown to the author if bounded martingales converge when they
are defined on a non $\sigma$-additive measure. To avoid that problem, 
we introduce, in the next section, a different less natural
rearrangement for the $u_{0,p_n}$, the shuffle, for which we finally prove a
convergence result.

\subsection{Shuffle rearrangement of two-scale limits}\label{subsect:RearrangementContinuous}

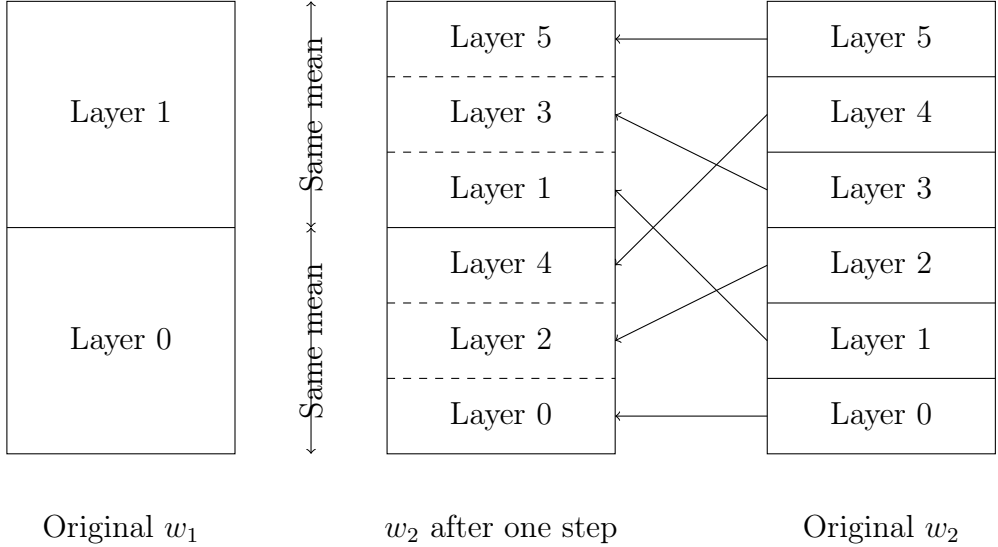
\begin{figure}[t]
\begin{center}
\begin{tikzpicture}
  \draw[xshift=0cm] (0,0) rectangle (3,6);
  \foreach \y in {3} {
    \draw[xshift=0cm]  (0,\y) -- (3,\y);
  }
  \node at (1.5,-1) {Original $w_1$ };
 \draw[xshift=5cm] (0,0) rectangle (3,6);
  \foreach \y in {3} {
    \draw[xshift=5cm]  (0,\y) -- (3,\y);
  }
  \foreach \y in {1,2,4,5} {
    \draw[dashed,xshift=5cm]  (0,\y) -- (3,\y);
  }
  \node[xshift=5cm] at (1.5,-1) {$w_2$ after one step};
 
\draw[<->] (4,0) -- (4,3);
\node[rotate=90] at (4,1.5) {Same mean}; 
\draw[<->] (4,3) -- (4,6);
\node[rotate=90] at (4,4.5) {Same mean};

\node[xshift=0cm] at  (1.5,1.5) {Layer $0$};
\node[xshift=0cm] at  (1.5,4.5) {Layer $1$};

  \draw[xshift=10cm] (0,0) rectangle (3,6);
  \foreach \y in {1,2,3,4,5} {
    \draw[xshift=10cm]  (0,\y) -- (3,\y);
  }
  \node[xshift=10cm] at (1.5,-1) {Original $w_2$};

\node[xshift=5cm] at  (1.5,0.5) {Layer $0$};
\node[xshift=5cm] at (1.5,1.5) {Layer $2$};
\node[xshift=5cm] at  (1.5,2.5) {Layer $4$};
\node[xshift=5cm] at  (1.5,3.5) {Layer $1$};
\node[xshift=5cm] at  (1.5,4.5) {Layer $3$};
\node[xshift=5cm] at  (1.5,5.5) {Layer $5$};

\node[xshift=10cm] at  (1.5,0.5) {Layer $0$};
\node[xshift=10cm] at  (1.5,1.5) {Layer $1$};
\node[xshift=10cm] at  (1.5,2.5) {Layer $2$};
\node[xshift=10cm] at  (1.5,3.5) {Layer $3$};
\node[xshift=10cm] at  (1.5,4.5) {Layer $4$};
\node[xshift=10cm] at  (1.5,5.5) {Layer $5$};

\draw[<-](8,0.5)--(10,0.5);
\draw[<-](8,1.5)--(10,2.5);
\draw[<-](8,2.5)--(10,4.5);
\draw[<-](8,3.5)--(10,1.5);
\draw[<-](8,4.5)--(10,3.5);
\draw[<-](8,5.5)--(10,5.5);

\end{tikzpicture}
\end{center}
\caption{One step of the measure preserving rearrangement $M_1=2$ and $M_2=6$}\label{fig:RearrangementOneStep}
\end{figure}

\begin{figure}[t]
\begin{center}
\begin{tikzpicture}
\draw[xshift=0cm] (0,0) rectangle (3,6);
  \foreach \y in {3} {
    \draw[xshift=0cm]  (0,\y) -- (3,\y);
  }
  \foreach \y in {1,2,4,5} {
    \draw[dashed,xshift=0cm]  (0,\y) -- (3,\y);
  }

 \draw[xshift=5cm] (0,0) rectangle (3,6);
  \foreach \y in {1,2,3,4,5} {
   \draw[xshift=5cm]  (0,\y) -- (3,\y);
 }
  \foreach \y in {0.5,1.5,2.5,3.5,4.5,5.5} {
   \draw[dashed,xshift=5cm]  (0,\y) -- (3,\y);
 }

 \draw[xshift=10cm] (0,0) rectangle (3,6);
  \foreach \y in {0.5,1,1.5,2,2.5,3,3.5,4,4.5,5,5.5} {
    \draw[xshift=10cm]  (0,\y) -- (3,\y);
  }

\draw[<-](3,0.5)--(5,0.5);
\draw[<-](3,1.5)--(5,2.5);
\draw[<-](3,2.5)--(5,4.5);
\draw[<-](3,3.5)--(5,1.5);
\draw[<-](3,4.5)--(5,3.5);
\draw[<-](3,5.5)--(5,5.5);

\draw[<-](8,0.25)--(10,0.25);
\draw[<-](8,0.75)--(10,3.25);
\draw[<-](8,1.25)--(10,0.75);
\draw[<-](8,1.75)--(10,3.75);
\draw[<-](8,2.25)--(10,1.25);
\draw[<-](8,2.75)--(10,4.25);
\draw[<-](8,3.25)--(10,1.75);
\draw[<-](8,3.75)--(10,4.75);
\draw[<-](8,4.25)--(10,2.25);
\draw[<-](8,4.75)--(10,5.25);
\draw[<-](8,5.25)--(10,2.75);
\draw[<-](8,5.75)--(10,5.75);

\node at (11.5, 0.25) {$0$};
\node at (11.5, 0.75) {$1$};
\node at (11.5, 1.25) {$2$};
\node at (11.5, 1.75) {$3$};
\node at (11.5, 2.25) {$4$};
\node at (11.5, 2.75) {$5$};
\node at (11.5, 3.25) {$6$};
\node at (11.5, 3.75) {$7$};
\node at (11.5, 4.25) {$8$};
\node at (11.5, 4.75) {$9$};
\node at (11.5, 5.25) {$10$};
\node at (11.5, 5.75) {$11$};

\node at (6.5, 0.25) {$0$};
\node at (6.5, 0.75) {$6$};
\node at (6.5, 1.25) {$1$};
\node at (6.5, 1.75) {$7$};
\node at (6.5, 2.25) {$2$};
\node at (6.5, 2.75) {$8$};
\node at (6.5, 3.25) {$3$};
\node at (6.5, 3.75) {$9$};
\node at (6.5, 4.25) {$4$};
\node at (6.5, 4.75) {$10$};
\node at (6.5, 5.25) {$5$};
\node at (6.5, 5.75) {$11$};

\node at (1.5, 0.25) {$0$};
\node at (1.5, 0.75) {$6$};
\node at (1.5, 1.25) {$2$};
\node at (1.5, 1.75) {$8$};
\node at (1.5, 2.25) {$4$};
\node at (1.5, 2.75) {$10$};
\node at (1.5, 3.25) {$1$};
\node at (1.5, 3.75) {$7$};
\node at (1.5, 4.25) {$3$};
\node at (1.5, 4.75) {$9$};
\node at (1.5, 5.25) {$5$};
\node at (1.5, 5.75) {$11$};

\node at (1.5,-1) {Second step};
\node at (6.5,-1) {First step};
\node at (11.5,-1) {Original};
\end{tikzpicture}
\end{center}
\caption{Two steps of the measure preserving rearrangement $M_1=2$,
  $M_2=6$ and $M_3=12$}\label{fig:CompleteRearrangement}
\end{figure}
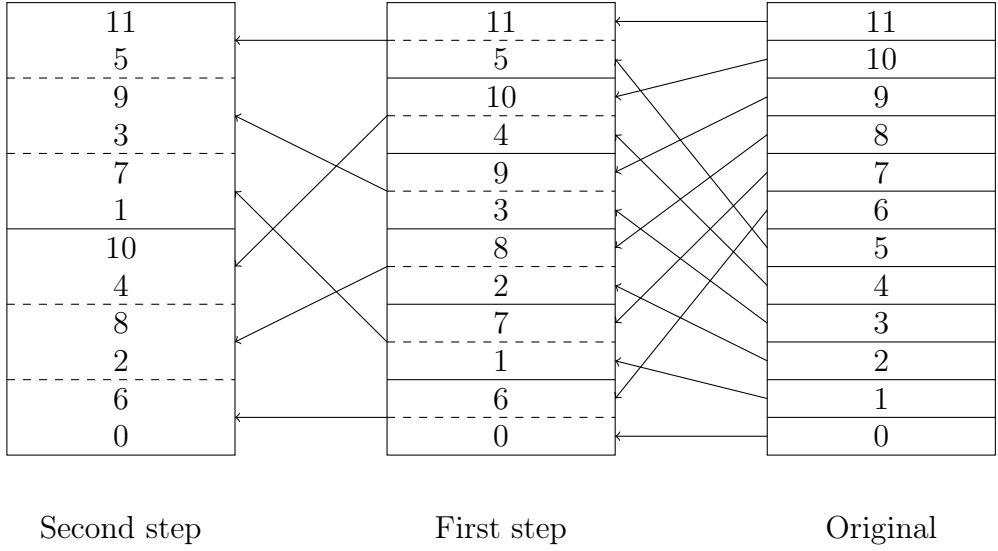

In the previous section, we investigated a rearrangement where the set
$\Omega\cart\lbrack0,p_n\lbrack^d$
was subdivided into $M_n^d$ subsets indexed by
$\bm{\alpha}\in\mathbb{Z}^d$.
In this section, we finally construct a rearrangement, the shuffle, that results in a
bounded martingale; thus establishing a convergence result for the
$p_n$-two-scale limits  as $n$ tends to $+\infty$. To do
so,
we replace the variable $\bm{\alpha}$ belonging to $\mathbb{Z}^d$
 with the variable $\bm{y}'$ that belongs to $\lbrack0,1\lbrack^d$. Like the
 variable $\bm{\alpha}$ of the previous section, the variable
 $\bm{y}'$ indicates which hypercube of edge length $p_0$
we consider. The variable $\bm{y}$ remains unchanged and continue to
represent the location inside the hypercube indexed by $\bm{y}'$.
We set for $\bm{x}$ in $\Omega$, $\bm{y}$ in $\lbrack0,p_0\rbrack^d$
and $\bm{y}'$ in $\lbrack0,1\rbrack^d$,
\begin{equation*}
w_{n}(\bm{x},\bm{y},\bm{y}'):=v_{p_n}(\bm{x},\bm{\alpha}(\bm{y}'),\bm{y}),
\end{equation*}
where $\bm{\alpha}(\bm{y}')_i=\lfloor M_n\bm{y}'_i\rfloor$ for all integers $i$ in
$\llbracket 1,d\rrbracket$.  
Using Proposition~\ref{prop:RearrangementIsImbricated}, we derive that
for almost all $\bm{x}$ in $\Omega$, $\bm{y}$ in $\CubeHomL_{p_0}$,
$(j,n)$ in $\mathbb{N}^2$, and $\bm{\alpha}$ in $\llbracket0,M_n-1\rrbracket^d$
\begin{multline}
\int_{\prod_{i=1}^d\lbrack\frac{\alpha_i}{M_n},\frac{\alpha_i+1}{M_n}
  \lbrack} w_{n}(\bm{x},\bm{y},\bm{y}') \ud\bm{y}'
=\\=
\sum_{\bm{\beta}\in\llbracket0,\frac{M_{n+j}}{M_n}-1\rrbracket^d}
\int_{\prod_{i=1}^d\lbrack\frac{M_n\beta_i+\alpha_i}{M_{n+j}},\frac{M_n\beta_i+\alpha_i+1}{M_{n+j}}\lbrack} w_{n+j}(\bm{x},\bm{y},\bm{y}') \ud\bm{y}'.
\end{multline}

To transform the $w_n$ into martingales, we need to shuffle the
hypercubes as in Figure~\ref{fig:RearrangementOneStep} 
where, to simplify the drawing, 
homogenization was only performed on the last component of
$\mathbb{R}^d$, hence the presence of layers instead of hypercubes. In
that figure, we show one step of the rearrangement. 
As seen in the drawing, each step of the 
rearrangement is measure preserving, therefore the full rearrangement
is also measure preserving. We need $n-1$ such steps to fully
rearrange $w_{n}$. 

To define rigorously this rearrangement, we begin by defining the
function that maps the rearranged layer index onto the unrearranged
layer index:
\begin{subequations}\label{subeq:defHn}
\begin{equation}
\begin{split}
R_{M,m}:&\llbracket0,Mm-1\rrbracket\to \llbracket0,Mm-1\rrbracket\\
i&\mapsto M\cdot(i \mod m)+\left\lfloor\frac{i}{m}\right\rfloor.
\end{split}
\end{equation}
The application $R_{M,m}$ maps $km+j$ to $jM+k$ when $k$ belongs
to $\llbracket0, M-1\rrbracket$ and $j$ belongs 
to $\llbracket0, m-1\rrbracket$. We also have
$R_{M,m}\circ R_{m,M}=R_{m,M}\circ R_{M,m}=\mathrm{Id}$.

Then, we set the function that maps the rearranged layer onto the
unrearranged one:
\begin{equation}
\begin{split}
  h^{*}_{M,m}:\lbrack0,1\lbrack&\to\lbrack0,1\lbrack,\\
  y'&\mapsto \frac{R_{M,m}(\lfloor Mmy'\rfloor)}{Mm}+
\left(y' -\frac{\lfloor Mmy'\rfloor)}{Mm}\right).
\end{split}
\end{equation}

This represents only one step of the rearrangement on one component.
For hypercubes, the permutation is the same but is done componentwise:
we set
\begin{equation*}
\begin{split}
h_{M,m}\mathpunct{:}\rbrack0,1\lbrack^d&\to \rbrack0,1\lbrack^d,\\
(y_1',\dotsc,y_n')&\mapsto(h^{*}_{M,m}(y_1'),\dotsc,h^{*}_{M,m}(y_n')).
\end{split}
\end{equation*}
And obtain one step of the rearrangement on all $d$ components.
For the complete rearrangement on one component, see
Figure~\ref{fig:CompleteRearrangement}, 
we set
\begin{equation}
H^{*}_n:= h^{*}_{M_{n-1},m_n}\circ\dotsc\circ h^{*}_{M_1,m_2}\circ
h^{*}_{M_0,m_1}.
\end{equation}
To get the complete rearrangement on all components we set
\begin{equation}
\begin{split}
H_n\mathpunct{:}\lbrack 0,1\lbrack^d&\to \lbrack0,1\lbrack^d,\\
(y_1',\dotsc,y_n')&\mapsto(H^{*}_n(y_1'),\dotsc,H^{*}_n(y_n')).
\end{split}
\end{equation} 
\end{subequations}
We also have
\begin{equation*}
H_n= h_{M_{n-1},m_n}\circ\dotsc\circ h_{M_1,m_2}\circ
h_{M_0,m_1}.
\end{equation*}
The function $H_n$ shuffles the hypercubes
$\prod_{i=1}^d\lbrack\beta_i/M_n,(\beta_i+1)/M_n\lbrack$, hence we call
$H_n$ the shuffle function.

Finally, we define
\begin{equation}\label{eq:predefwidetildewn}
\widetilde{w}_{n}(\bm{x},\bm{y},\bm{y}'):=w_{n}(\bm{x},\bm{y},H_n(\bm{y}')).
\end{equation} 
This measure preserving rearrangement, the shuffle, is purposefully constructed so the $\widetilde{w}_n$ form a martingale for the 
following filtration of $\sigma$-fields
$\mathcal{F}_n=\mathcal{B}(\Omega)\cart\mathcal{B}(\lbrack0,p_0\rbrack^d)\cart
\sigma\left\{\prod_{i=1}^d
  \left\lbrack\frac{\beta_i}{M_n},\frac{\beta_i+1}{M_n}\right\lbrack,
  \bm{\beta}\in\llbracket 0,M_n-1\rrbracket^d\right\}$.

\begin{remark}
The above rearrangement of hypercubes is similar to the one used for
computing in place the Discrete Fast Fourier Transform: the bit
reversal. In the special case where $M_n=2^n$, the rearrangement simply
exchanges layers $i$, \textit{i.e.}, $\lbrack i/2^n,(i+1)/2^n\lbrack$, and $i'$,
\textit{i.e.}, $\lbrack i'/2^n,(i'+1)/2^n\lbrack$, when $i$ and $i'$ are
 bit reversal permutations of each other. \textit{I.E} when
 $i=\sum_{j=0}^{n-1}b_j2^j$ and $i'=\sum_{j=0}^{N-1}b_j2^{n-1-j}$.
\end{remark}

\begin{remark}\label{remark:explicitHnetoile}
For general $M_n$, the rearrangement of hypercubes is also a bit
reversal but for a mixed basis. If
$\lfloor M_ny'\rfloor=\sum_{j=1}^nb_{j}M_{j-1}$ with $b_j$ in
$\llbracket0,m_j-1\rrbracket$, then
\begin{equation*}
H^*_n\left(\frac{1}{M_n}\sum_{j=1}^nb_j\frac{M_n}{M_j}
+(y'-\frac{\lfloor M_ny'\rfloor}{M_n})\right)=y'.
\end{equation*}
\end{remark}

We now state our main result as a self contained theorem.
\begin{theorem}[Two-Scale Shuffle convergence]\label{theo:TwoScaleLimitsMartingale}
Let $\Omega$ be a bounded open domain of $\mathbb{R}^d$ with $d\geq1$.
Let $(u_\varepsilon)_{\varepsilon>0}$ be a bounded sequence of functions
belonging to $\mathrm{L}^2(\Omega)$. Let $(p_n)_{n\in\mathbb{N}}$ be an increasing
sequence of positive numbers that satisfy Assumption~\ref{assum:IntergerScaleFactors}.
Set for all $n\geq0$ $M_n:=p_n/p_0$ and
for all $n\geq1$ $m_n:=p_n/p_{n-1}$. Let
$(\varepsilon_k)_{k\in\mathbb{N}}$ be a decreasing sequence of
positive real numbers converging to $0$ such that the sequence
$(u_{\varepsilon_k})_{k\in\mathbb{N}}$ $p_n$-two-scale converges to
$u_{0,p_n}$ for all non-negative integer $n$.

Set
\begin{equation*}
\begin{split}
\widetilde{w}_n:\Omega\cart\lbrack0,p_0\rbrack^d\cart\lbrack0,1\rbrack^d&\to\mathbb{R}\\
(\bm{x},\bm{y},\bm{y}')&\mapsto u_{0,p_n}\left(\bm{x}, p_0\lfloor M_nH_n(\bm{y}')\rfloor+\bm{y}\right).
\end{split}
\end{equation*}
where $H_n$ is defined by Equations~\eqref{subeq:defHn}.

Then, the sequence $\widetilde{w}_n$ is a bounded martingale in 
$\mathrm{L}^2(\Omega\cart\lbrack0,p_0\rbrack^d\cart\lbrack0,1\rbrack^d)$,
in the sense of Definition~\ref{defin:martingales},
for the filtration 
\begin{equation}\label{eq:FiltrationRearrangedContinuous}
\mathcal{F}_n=\mathcal{B}(\Omega)\cart\mathcal{B}(\lbrack0,p_0\rbrack^d)\cart
\sigma\left\{\prod_{i=1}^d \left\lbrack\frac{\beta_i}{M_n},\frac{\beta_i+1}{M_n}\right\lbrack,\bm{\beta}\in\llbracket0,M_n-1\rrbracket^d\right\}.
\end{equation}
And, the sequence $\widetilde{w}_n$ converges both strongly in
$\mathrm{L}^2(\Omega\cart\lbrack0,p_0\rbrack^d\cart\lbrack0,1\rbrack^d)$
and almost everywhere in $\Omega\cart\lbrack0,p_0\rbrack^d\cart\lbrack0,1\rbrack^d$
to $\widetilde{w}_\infty$, which we call the Two-Scale Shuffle limit. Moreover,
\begin{equation*}
\iiint_A\widetilde{w}_n(\bm{x},\bm{y},\bm{y}')\ud\bm{y}'\ud\bm{y}\ud\bm{x}=\iiint_A\widetilde{w}_\infty(\bm{x},\bm{y},\bm{y}')\ud\bm{y}'\ud\bm{y}\ud\bm{x},
\end{equation*}
for all sets $A$ in $\mathcal{F}_n$. \textit{I.E.}, by Definition~\ref{defin:ConditionalExpectation},
$\widetilde{w}_n=\mathbb{E}(\widetilde{w}_\infty\vert \mathcal{F}_n)$.
\end{theorem}
\begin{proof}
The $\widetilde{w}_n$ were constructed specifically so as to be a
martingale for the
filtration~\eqref{eq:FiltrationRearrangedContinuous}.
To prove they are a martingale for the filtration $(\mathcal{F}_n)_{n\in\mathbb{N}}$, we only need to prove that for all
non-negative integer $n$, for almost all $\bm{x}$ in $\Omega$, almost all
$\bm{y}$ in $\lbrack0,p_0\rbrack^{d}$ and for all $\bm{\beta}$ in
$\llbracket0,M_{n}-1\rrbracket^d$,  we have 
\begin{equation*}
\int_{\prod_{i=1}^d\lbrack\frac{\beta_i}{M_n},\frac{\beta_i+1}{M_n}\lbrack}\widetilde{w}_{n+1}(\bm{x},\bm{y},\bm{y}') \ud\bm{y}'
=
\int_{\prod_{i=1}^d\lbrack\frac{\beta_i}{M_n},\frac{\beta_i+1}{M_n}\lbrack}\widetilde{w}_n(\bm{x},\bm{y},\bm{y}') \ud\bm{y}'.
\end{equation*}
\textit{I.E.}, we need to show that
\begin{equation*}
\int_{\prod_{i=1}^d\lbrack\frac{\beta_i}{M_n},\frac{\beta_i+1}{M_n}\lbrack}w_n(\bm{x},\bm{y},H_n(\bm{y}')) \ud\bm{y}'
=\int_{\prod_{i=1}^d\lbrack\frac{\beta_i}{M_n},\frac{\beta_i+1}{M_n}\lbrack}w_{n+1}(\bm{x},\bm{y},H_{M_n,m_{n+1}}\circ
H_n(\bm{y}')) \ud\bm{y}'.
\end{equation*}
But $H_n$ maps any hypercube $\prod_{i=1}^d\lbrack\beta_i/M_n,(\beta_i+1)/M_n\lbrack$
to another hypercube  $\prod_{i=1}^d\lbrack\beta'_i/M_n,(\beta'_i+1)/M_n\lbrack$ and
$H_n$ is measure preserving. Therefore, we only need to prove that 
for almost all $\bm{x}$ in $\Omega$, almost all
$\bm{y}$ in $\lbrack0,p_0\rbrack^{d}$ and for all $\bm{\beta}$ in
$\llbracket0,M_{n}-1\rrbracket^d$
\begin{equation*}
\int_{\prod_{i=1}^d\lbrack\frac{\beta_i}{M_n},\frac{\beta_i+1}{M_n}\lbrack}w_n(\bm{x},\bm{y},\bm{y}') \ud\bm{y}'
=
\int_{\prod_{i=1}^d\lbrack\frac{\beta_i}{M_n},\frac{\beta_i+1}{M_n}\lbrack}w_{n+1}(\bm{x},\bm{y},H_{M_n,m_{n+1}}(\bm{y}')) \ud\bm{y}'.
\end{equation*}
is satisfied. But this equality is equivalent to 
\begin{multline*}
\frac{1}{M_n^d}v_{p_n}(\bm{x},\bm{\beta},\bm{y})
=\\=
\frac{1}{M_{n+1}^d}\sum_{\bm{\beta}'\in\llbracket0,m_{n+1}-1\rrbracket^d}
v_{p_{n+1}}(\bm{x},R_{M_n,m_{n+1}}(m_{n+1}\bm{\beta}+\bm{\beta'}),\bm{y})
=\\=
\frac{1}{M_{n+1}^d}\sum_{\bm{\beta}'\in\llbracket0,m_{n+1}-1\rrbracket^d}
v_{p_{n+1}}(\bm{x},(\bm{\beta}+\bm{\beta'}M_n),\bm{y}),
\end{multline*}
which is true by
Proposition~\ref{prop:RearrangementIsImbricated}. Therefore, the
sequence $\widetilde{w}_n$ is a martingale for the filtration $(\mathcal{F}_n)_{n\in\mathbb{N}}$.

By Proposition~\ref{prop:TwoScaleLimitL2Bound}, this martingale is bounded in $\mathrm{L}^2$.
It converges both strongly in $\mathrm{L}^2(\Omega\cart\lbrack0,p_0\rbrack^d\cart\lbrack0,1\rbrack^d)$ 
and almost everywhere to a function $\widetilde{w}_\infty$,
see~\cite[Corollary~7.22]{Kallenberg:FoundationsProbability}.
\end{proof}

\begin{corollary}\label{corrol:RecoveryTwoScale}
It is possible to recover $u_{0,p_n}$ from the Two-Scale Shuffle limit
$\widetilde{w}_{\infty}$. First, for all $\bm{\beta}$ in
$\llbracket0,M_{n}-1\rrbracket^d$, all $\bm{y}'$ in
$\prod_{i=1}^d\lbrack\beta_i/M_n,(\beta_i+1)/M_n\lbrack$, and almost all
$(\bm{x},\bm{y})$
in $\Omega\cart\lbrack0,p_0\rbrack^d$, we have
\begin{equation*}
\widetilde{w}_n(\bm{x},\bm{y},\bm{y}')=M_n^d
\int_{\prod_{i=1}^d\lbrack\frac{\beta_i}{M_n},\frac{\beta_i+1}{M_n}\lbrack}\widetilde{w}_\infty(\bm{x},\bm{y},\bm{y}')\ud\bm{y}'
\end{equation*}
because
$\widetilde{w}_n=\mathbb{E}(\widetilde{w}_\infty\vert\mathcal{F}_n)$.
Since the shuffle function $H_n$ is one to one from $\lbrack0,1\lbrack^d$ to
$\lbrack0,1\lbrack^d$, see Remark~\ref{remark:explicitHnetoile}, we have 
$w_n(\bm{x},\bm{y},\bm{y}') =\widetilde{w}_n(\bm{x},\bm{y}, H_n^{-1}(\bm{y}'))$.
Finally, $u_{0,p_n}(\bm{x},\bm{y})$ is equal to the
constant value taken by
$\bm{y}'\mapsto w_n(\bm{x},\bm{y}- p_0\lfloor{\bm{y}}/{p_0}\rfloor, \bm{y}')$
when $\bm{y}'$ belongs to the hypercube
$\lbrack\lfloor{\bm{y}}/{p_0}\rfloor/M_n,(\lfloor{\bm{y}}/{p_0}\rfloor+1)/{M_n}\lbrack$.
\end{corollary}

\section{Application: heat equation in
  multilayers}\label{sect:application}
In this section, we consider the multilayer heat
equation with three spatial dimensions which we homogenize along the
vertical space variable, \textit{i.e.}, along the
direction perpendicular to the layers.

In~\cite{Santugini:HomogenizationMultilayerHeat}, the author
established the equations satisfied by the two-scale limits of the
heat equation in multilayers with transmission conditions between
adjacent layers. 
When the magnitude of the interlayer conductivity between
  adjacent layers is weak, see~\cite[\S6.1]{Santugini:HomogenizationMultilayerHeat}, 
 the two-scale limit depends 
on the number of layers present in the homogenization period,
\textit{i.e.}, on the scale factors $p_n$. For given values of the slow
variables $(\bm{x},t)$, the two scale limit is piecewise constant in its 
scalar fast variable $y$ and takes as many values as there are layers
in a single homogenization cell.
Our goal is to establish the equation satisfied by the limit of two-scale
limits, \textit{i.e.}, the Two-Scale Shuffle limit, as
  defined in Theorem~\ref{theo:TwoScaleLimitsMartingale}.

To do so, we first recall previously known results in
\S\ref{subsect:MultilayerHeatPreviouslyKnownResults}, then derive new
results using Two-Scale Shuffle Convergence in \S\ref{subsect:MultilayerHeatNewResults}.

\subsection{The two-scale limit of the multilayer
    heat equation}\label{subsect:MultilayerHeatPreviouslyKnownResults}
We start by recalling some results we obtained in~\cite{Santugini:HomogenizationMultilayerHeat}.
To avoid unnecessary complications, we consider here a simpler problem
than the one considered in~\cite[System (4.1)]{Santugini:HomogenizationMultilayerHeat}.
Let $\Omega$ be $B\cart \rbrack0,1\lbrack$ where $B$ is a convex bounded open subset of $\mathbb{R}^{2}$ with smooth
boundary. 
Let $\delta$, $0<\delta<1/2$. Let $I$ be the interval $\rbrack\delta,1-\delta\lbrack$. 
For all $N$, let $I_N$ be $\bigcup_{j=0}^{N-1}\rbrack(j+\delta)/N,(j+1-\delta)/N\lbrack$. 
Let
$\Omega^N$ be the domain $B\cart I_N$.
Let $\Gamma^{N,+}_j=B\cart
\{(j+\delta)/N\}$ and 
$\Gamma^{N,-}_j=B\cart \{(j-\delta)/N\}$.
Let $\Gamma^{N,+}=\bigcup_{j=1}^{N-1}\Gamma^{N,+}_j$
and $\Gamma^{N,-}=\bigcup_{j=1}^{N-1}\Gamma^{N,-}_j$.
Let $\Gamma_e=\partial B\cart\rbrack0,1\lbrack\cup\Gamma^{N,+}_0\cup\Gamma^{N,-}_N$.
Let $\gamma$ be the application on $\partial\Omega^N$ that maps $u$ in
$\mathrm{H}^1(\Omega^N)$ to its trace on $\partial\Omega^N$.
Let $\gamma'$ be the trace operator swapped between 
 $\Gamma^{N,+}_j$ and $\Gamma^{N,-}_j$: \textit{i.e.}, 
$\gamma'u(\hat{\bm{x}},(j\pm\delta)/N)=\gamma u(\hat{\bm{x}},(j\mp\delta)/N)$
for all $\hat{\bm{x}}$ in $B$ and
all $j$ in $\llbracket1,N-1\rrbracket$. 
See  Figure~\ref{fig:MultilayerHeatGeometry} where we schematized the
   three dimensional domain $\Omega^N$ by projecting it onto the two-dimensional plane. 
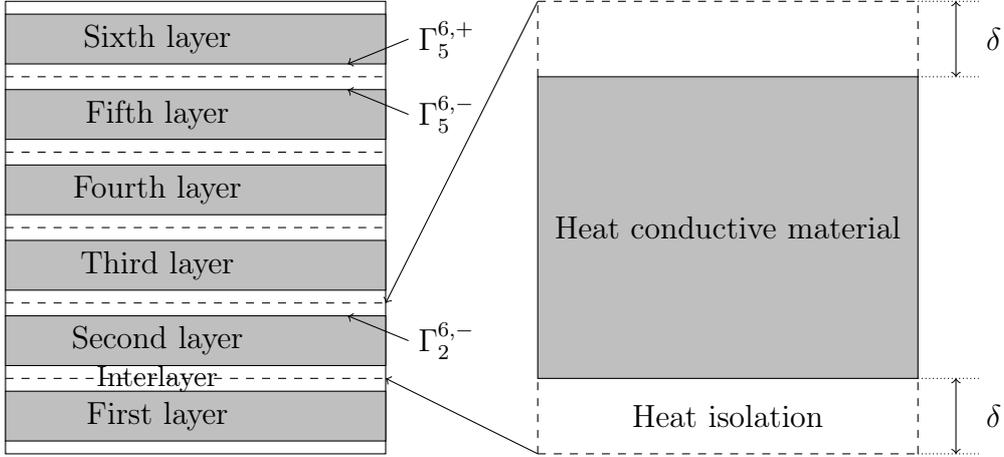
\begin{figure}
\begin{tikzpicture}
\draw (0,0) rectangle (5,6);
\foreach \y in {0,1,2,3,4,5} {
  \draw[fill=lightgray] (0,\y+0.17) rectangle (5,\y+0.83);
}
\foreach \y in {1,2,3,4,5} {
 \draw[dashed]  (0,\y) -- (5,\y);
}

\draw[<-] (4.5, 5.17) -- (5.3,5.5);
\node at (5.8,5.5) {$\Gamma^{6,+}_5$} ;
\draw[<-] (4.5, 4.83) -- (5.3,4.5);
\node at (5.8,4.5) {$\Gamma^{6,-}_5$} ;

\draw[<-] (4.5, 1.83) -- (5.3,1.5);
\node at (5.8,1.5) {$\Gamma^{6,-}_2$} ;

\node at (2,0.5) {First layer};
\node at (2,1) {{\small Interlayer}};
\node at (2,1.5) {Second layer};
\node at (2,2.5) {Third layer};
\node at (2,3.5) {Fourth layer};
\node at (2,4.5) {Fifth layer};
\node at (2,5.5) {Sixth layer};

\draw[->] (7,6) -- (5,2);
\draw[->] (7,0) -- (5,1);
\draw[fill=lightgray] (7,1) rectangle (12,5);
\node at (9.5,3) {Heat conductive material};
\node at (9.5,0.5) {Heat isolation};
 \foreach \y in {0,6} {
    \draw[dashed] (7,\y) -- (12,\y);
 }
 \draw[dashed] (7,0) -- (7,1);
 \draw[dashed] (7,5) -- (7,6);
 \draw[dashed] (12,0) -- (12,1);
 \draw[dashed] (12,5) -- (12,6);

 \draw[<->] (12.5,0) -- (12.5,1) ;
 \draw[<->] (12.5,5) -- (12.5,6) ;
 \draw[densely dotted] (12,0) -- (12.8,0) ;
\draw[densely dotted] (12,1) -- (12.8,1) ;
 \draw[densely dotted] (12,5) -- (12.8,5) ;
\draw[densely dotted] (12,6) -- (12.8,6) ;
\node at (13,5.5) {$\delta$};
  \node at (13,0.5) {$\delta$};
\end{tikzpicture}
\caption{The multilayer geometry, six layers: $N=6$.}\label{fig:MultilayerHeatGeometry}
\end{figure}

Let $A$, $K$ and $J$ be positive reals: $A$ represents the
  heat conductivity inside $\Omega^N$, and $J$ is the magnitude of the surfacic
interlayer conductivity.
For all positive integer $N$, we consider the multilayer heat equation 
\begin{subequations}\label{subeq:MultilayerHeat}
\begin{equation}
\frac{\partial u_N}{\partial t}-A\Lapl u_N=0\text{ in $\Omega^N\cart\mathbb{R}^+$}
\end{equation}
with the boundary conditions
\begin{equation}
A\frac{\partial u_N}{\partial\bm{\nu}}=
\begin{cases}
0& \text{on $\Gamma_e\cart\mathbb{R}^+$}\\
-\frac{K} {N}\gamma u_N +\frac{J}{N}(\gamma'u_N-\gamma u_N)  
&\text{on $(\Gamma^{N,+}\cup\Gamma^{N,-})\cart\mathbb{R}^+$,}
\end{cases}
\end{equation} 
and the initial condition
\begin{equation}
u_N(\cdot,0)=u_{N}^0.
\end{equation}
\end{subequations}
We also define the energy
\begin{equation*}
\begin{split}
E^N(v)&=\frac{A}{2}\int_{\Omega^N}\lVert\nabla v(\bm{x})\rVert^2\ud\bm{x}
+\frac{K}{2N}\int_{\Gamma^{N,+}\cup\Gamma^{N,-}}\lvert \gamma
v\rvert^2\ud\sigma(\bm{x})
\\&\phantom{=}
+\frac{J}{2N}\sum_{j=1}^{N-1}\int_B\left\lvert
v(\hat{\bm{x}},\frac{j+\delta}{N})-v(\hat{\bm{x}},\frac{j-\delta}{N})
\right\rvert^2\ud\hat{\bm{x}},
\end{split}
\end{equation*}
for all $v$ in $H^1(\Omega^N)$. We suppose
\begin{equation*}
\sup_{N}E^N(u_{N}^0)<+\infty,
\end{equation*}
and denote by $u_{0,M}^0$ the $M$-two-scale
limit of the initial conditions $u_{N}^0$. We have
the energy equality
\begin{equation*}
E^N(u_{N}(\cdot,T))
+\int_0^T\int_{\Omega^N}\left\lvert\frac{\partial u_N}{\partial t}\right\rvert^2\ud\bm{x}\ud t=E^N(u_{N}(\cdot,T)),
\end{equation*}
for all $T\geq0$.
Because of the 
energy bound, for all $\bm{x}$
in $\Omega$ and all $j$ in $\mathbb{Z}/M\mathbb{Z}$, the function
$y\mapsto u_{0,M}^0(\bm{x},y)$ is constant in the
interval $\rbrack j+\delta,j+1-\delta\lbrack$. Moreover, that function is 
$M$-periodic in $y$. 
We denote by $u_{0,M_n,j}^0(\bm{x})$ the value taken by the function
$y\mapsto u_{0,M}^0(\bm{x},y)$ 
 in the interval $\rbrack j+\delta,j+1-\delta\lbrack$.

Using two-scale convergence~\cite{Allaire:TwoScaleConvergence,Nguetseng:TheoryHomogenization} 
and its variant on periodic surfaces~\cite{NeussRadu:These,
 NeussRadu:ExtensionsTwoScaleConv,Allaire.Damlamian.Hornung:TwoScaleSurfacicConvergence},
the properties of the two-scale limit of
$(u^N)_{N\in\mathbb{N}}$, solutions to the multilayer heat
system~\eqref{subeq:MultilayerHeat} with $J\neq 0$,
were established
in~\cite[Theorem~6.1]{Santugini:HomogenizationMultilayerHeat}. 

Let $(M_n)_{n\geq0}$ be a sequence of positive integers such that $M_0=1$ and 
such that $M_{n+1}$ is always an entire multiple of $M_n$. 
For all $(\bm{x},t)$ in $\Omega\cart\mathbb{R}^+$, 
the $M_n$-two-scale limit $y\mapsto u_{0,M_n}(\bm{x},t,y)$
takes $M_n$ values: it is constant in each interval
$\rbrack j+\delta, j+1-\delta\lbrack$. 
For $j$ in $\mathbb{Z}/M_n\mathbb{Z}$, we note $u_{0,M_n,j}(\bm{x},t)$ the value
of $u_{0,M_n}(\bm{x},t,\cdot)$ in this interval. We have $u_{0,M_n,j+M_n}(\bm{x},t)=u_{0,M_n,j}(\bm{x},t)$.
These functions satisfy, for all $j$ in $\mathbb{Z}/M_n\mathbb{Z}$,
the weak formulation of
\begin{subequations}\label{subeq:MultilayerTwoScaleLimit}
\begin{multline}\label{eq:WeakFormulationTwoScaleLimit}
\frac{\partial u_{0,M_n,j}}{\partial t}-A\Lapl_{\mathrm{T}}
u_{0,M_n,j}+\frac{2K}{1-2\delta}u_{0,M_n,j}
+\\+\frac{J}{1-2\delta}(2u_{0,M_n,j}-u_{0,M_n,j+1}-u_{0,M_n,j-1})
=0,
\end{multline}
in $\Omega\cart\mathbb{R}^+$, and where $\Lapl_{\mathrm{T}}=\frac{\partial^2}{\partial 
x_1^2}+\frac{\partial^2}{\partial x_2^2}$, 
$\nabla_{\mathrm{T}}=\lbrack\frac{\partial}{\partial
  x_1},\frac{\partial}{\partial x_2}\rbrack^{\mathrm{T}}$. And
 with boundary conditions
\begin{equation}\label{eq:WeakFormulationTwoScaleLimitBoundary}
\frac{\partial u_{0,M_n,j}}{\partial\bm{\nu}}=0\text{ on $(\partial B \cart\rbrack0,1\lbrack)\cart\mathbb{R}^+$,}
\end{equation}
and initial condition
\begin{equation}\label{eq:WeakFormulationTwoScaleLimitInitial}
u_{0,M_n,j} (\cdot,0)=u_{0,M_n,j}^0\text{ in $B \cart\rbrack0,1\lbrack$.}
\end{equation}
\end{subequations}

We have recalled previously known results on the
  properties of the two-scale limits of the multilayer heat
  equation. In the next section, we establish the properties satisfied
  by the Two-Scale Shuffle limit of the multilayer heat equation.

\subsection{The two-scale shuffle limit of the multilayer
    heat equation}\label{subsect:MultilayerHeatNewResults}

We now use Theorem~\ref{theo:TwoScaleLimitsMartingale} to have the
two-scale limits themselves converge. We choose $M_n=2^n$ to avoid
complications at first. We establish the following:
\begin{theorem}\label{theo:ShuffleLimitHeatEquationMultilayers}
Let $\widetilde{w_\infty}$ be the Two-Scale Shuffle limit defined from $(u_{0,2^n})_{n\in\mathbb{N}}$ as in 
Theorem~\ref{theo:TwoScaleLimitsMartingale}. For all $(\bm{x},t)$ in
$\Omega\cart\mathbb{R}^+$,
and $y'$ in $\lbrack0,1\rbrack$. The function
$\widetilde{w}_\infty(\bm{x},t,\cdot,y')$ is constant
in the interval $\rbrack\delta,1-\delta\lbrack$. If we denote by
$\widetilde{w}_\infty(\bm{x},t,y')$ the value of
$\widetilde{w}_\infty(\bm{x},t,\cdot,y')$
inside the interval $\rbrack\delta,1-\delta\lbrack$, the Two-Scale Shuffle limit
$\widetilde{w}_\infty$ is a weak solution to:
\begin{subequations}
\begin{equation}\label{eq:WeakFormulationShuffleLimit}
\begin{split}
\frac{\partial\widetilde{w}_\infty}{\partial t}(\bm{x},t,y')-A\Lapl_{\mathrm{T}}
\widetilde{w}_\infty (\bm{x},t,y')+\frac{2K}{1-2\delta} \widetilde{w}_\infty
(\bm{x},t,y')
&\\
+\frac{J}{1-2\delta}(2\widetilde{w}_\infty (\bm{x},t,y')-
\widetilde{w}_\infty (\bm{x},t,\tau^+ (y'))-\widetilde{w}_\infty (\bm{x},t,\tau^-(y')))
=&0,
\end{split}
\end{equation}
in $\Omega\cart\mathbb{R}^+\cart\rbrack0,1\lbrack$, and
where, for all non-negative integers $j$:
\begin{align*}
\tau^+(y')&=y'+3\cdot2^{-(j+1)}-1
\text{ when $1-2^{-j}\leq  y'<1-2^{-(j+1)}$}, \\
\tau^-(y')&=y'-3\cdot2^{-(j+1)}+1
\text{ when $2^{-(j+1)}\leq y'<2^{-j}$}.
\end{align*}
with boundary conditions
\begin{equation}
\frac{\partial \widetilde{w}_\infty}{\partial\bm{\nu}}=0\text{ on $(\partial B \cart\rbrack0,1\lbrack)\cart\mathbb{R}^+\cart\rbrack0,1\lbrack$,}
\end{equation}
and initial condition
\begin{equation*}
\widetilde{w}_\infty(\cdot,0,\cdot)=\widetilde{w}_\infty^0.
\end{equation*}
where $\widetilde{w}_\infty^0$ is the Two-Scale Shuffle Limit of the
initial conditions and where, as an abuse of notations, we denote by
$\widetilde{w}_\infty^0(\bm{x},t,y')$ the constant value taken by 
$y\mapsto\widetilde{w}_\infty^0(\bm{x},t,y,y')$ in the interval $\rbrack\delta,1-\delta\lbrack$.
\end{subequations}
\end{theorem}
\begin{proof}
Let
$n$ in $\mathbb{N}$ and ${\beta}$ belongs to 
$\llbracket0,2^n-1\rrbracket$. Let the $w_n$ be defined as
  in~\S\ref{subsect:RearrangementContinuous},
and the $\widetilde{w}_n$
be defined from the $2^n$-two scale limits $u_{0,2^n}$ as in 
Theorem~\ref{theo:TwoScaleLimitsMartingale}. Both the
$\widetilde{w}_n$ and the $w_n$ are
functions defined over 
$(B\cart\rbrack0,1\lbrack)\cart\mathbb{R}^+\cart\rbrack0,1\lbrack\cart\rbrack0,1\lbrack$.
For all $(\bm{x},t)$ in $(B\cart\rbrack0,1\lbrack)\cart\mathbb{R}^+$, and $y'$ in
$\rbrack0,1\lbrack$, the application $y\mapsto \widetilde{w}_n(\bm{x},t,y,y')$ is
constant on $\rbrack\delta,1-\delta\lbrack$. As an abuse of notations,
let's also denote by  $\widetilde{w}_n(\bm{x},t,y')$ the value of 
$\widetilde{w}_n(\bm{x},t,y',y)$ when $y$ belongs to
$\rbrack\delta,1-\delta\lbrack$.
We use the same abuse of notations for the $w_n$.
Consider a test function $\varphi$ belonging to 
$\mathcal{C}^\infty(\overline{\Omega}\cart\mathbb{R}^+)$.  Set
 $\psi(\bm{x},t,y)=\varphi(x,t)\mathds{1}\{y\in\lbrack\frac{{\beta}}{2^n},\frac{{\beta}+1}{2^n}\lbrack\}$.
Let $Q_T=\Omega\cart\mathbb{R}^+$.
Then, since the $u_{0,2^n,j}$ satisfy the weak formulation
of~\eqref{subeq:MultilayerTwoScaleLimit}, we have:
\begin{equation*}
\begin{split}
\iint_{Q_T}\int_{\lbrack\frac{\beta}{2^n},\frac{\beta+1}{2^n}\lbrack}
\frac{\partial w_n}{\partial t}(\bm{x},t,y')\cdot\varphi(\bm{x},t)\ud
y'\ud\bm{x}\ud t
&\\
+A\iint_{Q_T}
\int_{\lbrack\frac{\beta}{2^n},\frac{\beta+1}{2^n}\lbrack}
\nabla_{\mathrm{T}} w_n(\bm{x},t,y')\cdot\nabla_{\mathrm{T}}\varphi(\bm{x},t)\ud y'\ud\bm{x}\ud t
&\\
+\frac{2K}{1-2\delta}
\iint_{Q_T}\int_{\lbrack\frac{\beta}{2^n},\frac{\beta+1}{2^n}\lbrack}
w_n(\bm{x},t,y')\cdot\varphi(\bm{x},t)\ud y'\ud\bm{x}\ud t
&\\
+\frac{J}{1-2\delta}\int_{\lbrack\frac{\beta}{2^n},\frac{\beta+1}{2^n}\lbrack}
2w_n(\bm{x},t,y')
\cdot\varphi(\bm{x},t)\ud y'\ud\bm{x}\ud t
&\\
-\frac{J}{1-2\delta}\int_{\lbrack\frac{\beta}{2^n},\frac{\beta+1}{2^n}\lbrack}
(w_n(\bm{x},t,y'+2^{-n})+w_n(\bm{x},t,y'-2^{-n}))
\cdot\varphi(\bm{x},t)\ud y'\ud\bm{x}\ud t&=0,
\end{split}
\end{equation*}
where, to simplify notations, we consider the function $w_n$ to be $1$-periodic in $y'$.
Therefore, for all $\beta$ in $\llbracket 0, 2^n-1\rrbracket$,
\begin{equation}\label{eq:formulationwidetilden}
\begin{split}
\iint_{Q_T}\int_{\lbrack\frac{\beta}{2^n},\frac{\beta+1}{2^n}\lbrack}
\frac{\partial \widetilde{w}_n}{\partial
  t}(\bm{x},t,y')\cdot\varphi(\bm{x},t)\ud y'\ud\bm{x}\ud t
&\\
+A\iint_{Q_T}
\int_{\lbrack\frac{\beta}{2^n},\frac{\beta+1}{2^n}\lbrack}
\nabla_{\mathrm{T}} \widetilde{w}_n(\bm{x},t,y')\cdot\nabla_{\mathrm{T}}\varphi(\bm{x},t)\ud y'\ud\bm{x}\ud t
&\\
+\frac{2K}{1-2\delta}
\iint_{Q_T}\int_{\lbrack\frac{\beta}{2^n},\frac{\beta+1}{2^n}\lbrack}
\widetilde{w}_n(\bm{x},t,y')\cdot\varphi(\bm{x},t)\ud y'\ud\bm{x}\ud t
&\\
+\frac{J}{1-2\delta}\iint_{Q_T}\int_{\lbrack\frac{\beta}{2^n},\frac{\beta+1}{2^n}\lbrack}
2\widetilde{w}_n(\bm{x},t,y') \cdot\varphi(\bm{x},t)\ud y'\ud\bm{x}\ud t
&\\
-\frac{J}{1-2\delta}\iint_{Q_T}\int_{\lbrack\frac{\beta}{2^n},\frac{\beta+1}{2^n}\lbrack}
\widetilde{w}_n(\bm{x},t,{H_n^*}^{-1}(H^*_n(y')+2^{-n}))
\cdot\varphi(\bm{x},t)\ud y'\ud\bm{x}\ud t
&\\
-\frac{J}{1-2\delta}\iint_{Q_T}\int_{\lbrack\frac{\beta}{2^n},\frac{\beta+1}{2^n}\lbrack}
\widetilde{w}_n(\bm{x},t,{H_n^*}^{-1}(H^*_n(y')-2^{-n}))
\cdot\varphi(\bm{x},t)\ud y'\ud\bm{x}\ud t.
&=0
\end{split}
\end{equation}
Here $H^*_n$ is simply the bit reversal of the first $n$
coefficients in the binary expansion. Thus:
\begin{equation}\label{eq:AdjacentPlusBinaryReversal} 
{H_n^*}^{-1}(H^*_n(y'+2^{-n}))=
\begin{cases} 
y'+3\cdot2^{-(j+1)}-1&\text{if $1-2^{-j}\leq
  y'<1-2^{-(j+1)}$}, \\&\text{for $0\leq j\leq n-1$},\\ 
y'-1+2^{-n}&\text{if $1-2^{-n}\leq y'<1$}.
\end{cases}
\end{equation} 
And
\begin{equation}\label{eq:AdjacentMinusBinaryReversal}
{H_n^*}^{-1}(H^*_n(y')-2^{-n})=
\begin{cases}
y'-3\cdot2^{-(j+1)}+1&\text{if $2^{-(j+1)}\leq
  y'<2^{-j}$},\\&\text{for  $0\leq j\leq n-1$},\\ 
y'+1-2^{-n}&\text{if $0\leq y'< 2^{-n}$}.
\end{cases}
\end{equation}
Since $\varphi(\bm{x},t)\mathds{1}\{y'\in\lbrack\frac{\beta}{2^n},\frac{\beta+1}{2^n}\lbrack\}$ is $\mathcal{F}_n$-measurable and 
$\widetilde{w}_n=\mathbb{E}(\widetilde{w}_\infty\vert\mathcal{F}_n)$,  Equality~\eqref{eq:formulationwidetilden}
remains valid after replacing $\widetilde{w}_n$ by
$\widetilde{w}_\infty$.
Therefore,
\begin{equation*}
\begin{split}
\iint_{Q_T}\int_{\lbrack\frac{\beta}{2^n},\frac{\beta+1}{2^n}\lbrack}
\frac{\partial \widetilde{w}_\infty}{\partial
  t}(\bm{x},t,y')\cdot\varphi(\bm{x},t)\ud y'\ud\bm{x}\ud t
&\\
+A\iint_{Q_T}
\int_{\lbrack\frac{\beta}{2^n},\frac{\beta+1}{2^n}\lbrack}
\nabla_{\mathrm{T}} \widetilde{w}_\infty(\bm{x},t,y')\cdot\nabla_{\mathrm{T}}\varphi(\bm{x},t)\ud y'\ud\bm{x}\ud t
&\\
+\frac{2K}{1-2\delta}
\iint_{Q_T}\int_{\lbrack\frac{\beta}{2^n},\frac{\beta+1}{2^n}\lbrack}
\widetilde{w}_\infty(\bm{x},t,y')\cdot\varphi(\bm{x},t)\ud y'\ud\bm{x}\ud t
&\\
+\frac{J}{1-2\delta}\int_{\lbrack\frac{\beta}{2^n},\frac{\beta+1}{2^n}\lbrack}
2\widetilde{w}_\infty(\bm{x},t,y') \cdot\varphi(\bm{x},t)\ud y'\ud\bm{x}\ud t
&\\
-\frac{J}{1-2\delta}\int_{\lbrack\frac{\beta}{2^n},\frac{\beta+1}{2^n}\lbrack}(
\widetilde{w}_\infty(\bm{x},t,\tau^+(y'))-\widetilde{w}_\infty(\bm{x},t,\tau^-(y')))
\cdot\varphi(\bm{x},t)\ud y'\ud\bm{x}\ud t
&=0
\end{split}
\end{equation*}
for all $n$ in $\mathbb{N}$ and $\beta$ in
$\llbracket1,2^n-2\rrbracket$. Choose $y'$ in $\rbrack0,1\lbrack$, for any positive
integer $n$, set $\beta=\lfloor 2^ny'\rfloor$ and take the limit in
the above equality divided by $2^{-n}$ as $n$ tends to $+\infty$. 
\end{proof}

If instead of setting $M_n=2^n$, we consider a general sequence $(M_n)_{n\in\mathbb{N}}$, the same
reasonning holds. When $M_n$ is $2^n$, the shuffling of layers is the
bit reversal of the first $n$ coefficients of the binary representation of
$y$, thus involutive. This is not the case for general $M_n$ and we
must use Remark~\ref{remark:explicitHnetoile}. Therefore, utmost care
must be taken to compute the analogues
of~\eqref{eq:AdjacentPlusBinaryReversal}
and~\eqref{eq:AdjacentMinusBinaryReversal}. We provide the limit
in the general case without proof. In that case,
we have
\begin{align*}
{H_n^*}^{-1}(H^*_n(y')+\frac{1}{M_n})&=
\begin{cases} 
y'-{\sum_{l=1}^j}\frac{1}{M_l}+\frac{1}{M_{j+1}}&\text{if $1-\frac{1}{M_j}\leq
  y'<1-\frac{1}{M_{j+1}}$,}\\&\text{for $0\leq j\leq n-1$},\\ 
y'-1+\frac{1}{M_n}&\text{if $1-\frac{1}{M_n}\leq y'<1$}.
\end{cases}\\
{H_n^*}^{-1}(H^*_n(y')-\frac{1}{M_n})&=
\begin{cases} 
y'+{\sum_{l=1}^j}\frac{1}{M_l}-\frac{1}{M_{j+1}}&\text{if $\frac{1}{M_j}\leq 
y'<\frac{1}{M_{j+1}}$}, \\
&\text{for $0\leq j\leq n-1$},\\  
y'+1-\frac{1}{M_n}&\text{if $0\leq y'<\frac{1}{M_n}$}.
\end{cases}
\end{align*}
and the limit equation~\eqref{eq:WeakFormulationShuffleLimit} remains
valid if we set instead
\begin{subequations}
\begin{align}
\tau^+(y')&=y'-\sum_{l=1}^j\frac{1}{M_l}+\frac{1}{M_{j+1}}\text{ when $1-\frac{1}{M_j}\leq
  y'<1-\frac{1}{M_{j+1}}$},\\
\tau^-(y')&=y'+\sum_{l=1}^j\frac{1}{M_l}-\frac{1}{M_{j+1}}\text{ when $\frac{1}{M_{j+1}}\leq
y'<\frac{1}{M_{j}}$},
\end{align}
\end{subequations}
 for all non-negative integer $j$.

\section{Conclusion}
We have proven in this paper, see Theorem~\ref{theo:TwoScaleLimitsMartingale}, that the two-scale limits of a given
sequence of functions, computed for periods that are entire multiple of the
previous ones, form a bounded martingale and thus converge both
strongly in $\mathrm{L}^2$ and almost
everywhere. From the limit, called the Two-Scale Shuffle limit,
one can recover any element in the
sequence of two-scale limits: this limit contains all the information contained in the whole
sequence of two-scale limits, see
Corollary~\ref{corrol:RecoveryTwoScale}. For a 
good choice of increasing periods,
this limits captures everything that happens at any length scale that
is an entire multiple of $\varepsilon$. 

Unfortunately, this limit  does not
capture all phenomena with a period linear in $\varepsilon$:
it cannot capture phenomena with an irrational scale factor. The
construction of the martingale depends on the
assumption that $p_{n+1}$  is always an entire multiple of $p_n$. If there are
two interesting scales whose ratio is irrational then no choice of
periodic scale carry the information for both scales.  

We applied the notion of the Two-Scale Shuffle limit to the heat
equation on multilayers with transmission conditions between adjacent
layers. We then considered the solutions to these equations and 
established the equation satisfied by
their Two-Scale Shuffle limit.

To establish the convergence of the two-scale limit, we 
 used the shuffle of hypercubes described in
\S\ref{subsect:RearrangementContinuous}. Unfortunately, because of
this shuffle, it is not easy to reach an intuitive understanding of
the Two-Scale Shuffle limit. Results on the existence of the limit in the
setting of \S\ref{subsect:RearrangementSequence}  would not have that
drawback.
 Solving Problem~\ref{problem:ImbricatedPeriodicConvergence}
would be a first step to obtain a limit in this setting.

\bibliographystyle{plain}
\bibliography{maths,physic,ajoutmaths,ajoutphysic}
\end{document}